\documentclass[11pt]{amsart}

\usepackage{amsmath, amsthm, amssymb}
\usepackage{nicefrac,mathtools}

\usepackage{comment}

\usepackage[abs]{overpic}		  

\usepackage{xcolor}
\definecolor{indigo}{rgb}{0.29, 0.0, 0.51}  
\usepackage[colorlinks, urlcolor=indigo, linkcolor=indigo, citecolor=indigo]{hyperref}

\theoremstyle{plain}
\newtheorem{theorem}{Theorem}

\newtheorem{proposition}[theorem]{Proposition}
\newtheorem{conjecture}[theorem]{Conjecture}
\newtheorem{lemma}[theorem]{Lemma}
\newtheorem{question}[theorem]{Question}

\theoremstyle{definition}

\theoremstyle{remark}
\newtheorem{remark}[theorem]{Remark}

\numberwithin{theorem}{section}

\newcommand{\dfn}[1]{{\em #1}}        
\newcommand{\R}{\mathbb{R}}           
\newcommand{\Q}{\mathbb{Q}}           
\newcommand{\Z}{\mathbb{Z}}           

\makeatletter
\newcommand*\bigcdot{\mathpalette\bigcdot@{0.6}}
\newcommand*\bigcdot@[2]{\mathbin{\vcenter{\hbox{\scalebox{#2}{$\m@th#1\bullet$}}}}}
\makeatother

\DeclareMathOperator\tb{tb}                   
\DeclareMathOperator\tbb{\overline {\tb}}     
\DeclareMathOperator\rot{rot}                 

\begin{document}

\title{On contact cosmetic surgery} 

\author{John B. Etnyre}
\address{Department of Mathematics \\ Georgia Institute of Technology \\ Atlanta \\ Georgia}
\email{etnyre@math.gatech.edu}

\author{Tanushree Shah}
\address{Department of Mathematics \\ Chennai Mathematical Institute \\ India}
\email{tanushrees@cmi.ac.in}


\begin{abstract}
		We demonstrate that the contact cosmetic surgery conjecture holds for all non-trivial Legendrian knots, with the possible exception of Lagrangian slice knots. We also discuss the contact cosmetic surgeries on Legendrian unknots and make the surprising observation that some Legendrian unknots have a contact surgery with no cosmetic pair, while all other contact surgeries are contactomorphic to infinitely many other contact surgeries on the unknot.  
\end{abstract}

\maketitle

\section{Introduction}
In this note we establish that a contact analog of the cosmetic surgery conjecture holds for all Legendrian knots except possibly for a small family of Legendrian knots. 

We first recall the smooth cosmetic surgery conjecture. Given a knot $K$ in $S^3$, we say that two surgeries on $K$, say $S^3_K(r)$ and $S^3_K(r')$, are \dfn{cosmetic} if $S^3_K(r)$ and $S^3_K(r')$ are diffeomorphic and \dfn{truly cosmetic} if $S^3_K(r)$ and $S^3_K(r')$ are orientation preserving diffeomorphic. There are many examples of cosmetic surgeries on knots, but the only known truly cosmetic surgeries are on the unknot.  Thus the cosmetic surgery conjecture postulates that non-trivial knots admit no truly cosmetic surgeries (see \cite{Gordon1991} Conjecture 6.1). There has been a great deal of work on this conjecture \cite{DaemiEismeierLidman24pre, Hanselman2023,NiWu2015, OzsvathSzabo2011, Tao2019Pre, Wang2006, Wu2011}, and this paper heavily depends on that work. We recall the specific results we will need in Section~\ref{background}.  

Turning to contact geometry, we recall that given a Legendrian knot $L$ in a contact manifold $(M,\xi)$ contact $(r)$-surgery on $L$ is the result of removing a standard neighborhood of $L$ from $M$ and replacing it with a solid torus whose meridian has slope $r+\tb(L)$ and extending the contact structure over this torus to be any tight contact structure. See Section~\ref{background} for more details, but we note now that for any $r\not=1/n$ there will be more than one possibility for a contact surgery. 

We can now say that two contact surgeries on a Legendrian knot with different smooth surgery coefficients are \dfn{cosmetic} if there is a contactomorphism between the resulting manifolds. We note that since a contactomorphism between two contact manifolds is automatically orientation preserving, we do not need to distinguish between ``cosmetic" and ``truly cosmetic" as in the smooth case. 
Generalizing the cosmetic surgery conjecture to the contact category, we have the contact cosmetic surgery conjecture.
\begin{conjecture}[Contact cosmetic surgery conjecture]
Any Legendrian knot in $S^3$ with its standard tight contact structure that is not smoothly an unknot admits no cosmetic contact surgeries. 
\end{conjecture}
We note that the contact cosmetic surgery conjecture is a natural generalization of the smooth cosmetic surgery conjecture and is particularly interesting in light of trying to understand different contact surgery representations of different contact manifolds. In addition, our main theorem gives a small class of knots on which the conjecture might be false, and thus gives prime candidates to consider for the general cosmetic surgery conjecture too. 

Our main result is a proof of this conjecture for almost all Legendrian knots. 
\begin{theorem}\label{contactcosmeticthm}
The contact cosmetic surgery conjecture holds true for all non-trivial Legendrian knots except possibly for $\pm 2$ surgery on a Legendrian knot $L$ that is Lagrangian slice and is in a knot type $K$ with $\tau(K)=0, \tbb(K)=-1$, has Seifert genus $2$, and is quasi-positive. 
\end{theorem}
This theorem will be proven by constructing specific contact $(\pm 1)$-surgery diagrams for specific contact surgeries and then analyzing the $d_3$-invariant of the corresponding diagrams. See Section~\ref{obstruction} for the details and the Appendix for details on the linear algebra necessary for the computations. 

We note that since there are many possible contact $(r)$-surgeries on a Legendrian knot $L$ if $r\not=1/n$, there are two alternate versions of the contact surgery conjecture. We say that two different contact $(r)$-surgeries on a Legendrian knot $L$ are \dfn{weakly-cosmetic} if they are contactomorphic, and we say they are \dfn{strongly-cosmetic} if they are isotopic. We note that since the smooth surgery coefficient is the same for both contact surgeries, the manifolds obtained from surgery are canonically identified, so we can talk about isotopy of the contact structure. 
We do not conjecture an analog of the contact cosmetic surgery conjecture but do ask the following. 
\begin{question}
Which Legendrian knots in $S^3$ with its standard tight contact structure admit strongly-cosmetic or weakly-cosmetic surgeries?
\end{question}
We make a few simple observations about this question. 
\begin{proposition}\label{weak}
The following holds for a Legendrian knot in a non-trivial knot type.
\begin{enumerate}
\item There are no strongly-cosmetic surgeries on $L$ with smooth surgery coefficient $r<\tb(L)$. 
\item If contact $(+1/n)$-surgery for some $n>1$, respectively $n=1$, on $L$ has a non-vanishing contact invariant, then there are no strongly-cosmetic surgeries on $L$ for contact $(r)$-surgeries with $r>0$, respectively $r\geq 1$.
\item If $L$ is an L-space knot with $\tb(L)=2\tau(L)-1$, then there are no strongly-cosmetic surgeries on $L$. 
\item For smooth surgery coefficient $-2$ a Legendrian knot $L$ has no strongly-cosmetic surgeries, except possibly when $\rot(L)=0$, and has no weakly-cosmetic surgeries if $\tb(L)\leq -2$ except possibly when $\rot(L)=0$. 
\end{enumerate}
\end{proposition}
We note that Item~(3) in the theorem holds for many Legendrian knots. For example, all maximal Thurston-Bennequin invariant positive torus knots satisfy the condition. 
\begin{remark}
In contrast to Item~(4) above, we see that for some Legendrian knots, there are many weakly-cosmetic surgeries. Consider $\R^3$ with its standard contact structure $\xi_{std}=\ker(dz-y\, dx)$. The ``mirror map'' $m: \R^3\to \R^3:(x,y,z)\mapsto (x,-y,-z)$ is an orientation-preserving diffeomorphism that is smoothly isotopic to the identity but changes the coorientation on $\xi_{std}$. If $L$ is a Legendrian knot in $(\R^3,\xi_{std})$ such that $m(L)$ is Legendrian isotopic to $L$, then there is a contactomorphism of the complement of a neighborhood, denoted $S^3_L$, to itself that reverses the coorientation on the contact structure. If $\xi$ is a tight contact structure on a solid torus $S^1\times D^2$ that can be glued into $S^3_L$ to affect contact surgery on $L$, then we can also mirror $(S^1\times D^2, \xi)$ by sending $(\phi, x,y)$ to $(-\phi, x, -y)$ to get $(S^1\times D^2, \xi')$. The relative Euler class of the two contact structures have opposite sign, so if the relative Euler class of $\xi$ is not zero, we can use $\xi$ and $\xi'$ to perform two different contact surgeries on $L$ by gluing them to $S^3_L$. Now, using the mirror map $m$ on $S^3_L$ and the mirror on the solid torus, we can see that these two different contact surgeries give contactomorphic contact structures. 

There are many examples of knots $L$ satisfying the above condition, such as any positive Legendrian torus knot with $\rot=0$. We would like to thank the referee for pointing out this example. 
\end{remark}

Just as in the smooth setting, it is important to exclude Legendrian unknots from this conjecture, as they do admit cosmetic surgeries. We can explicitly write out all the contact cosmetic surgeries on Legendrian unknots. 
To state these results, we first recall the situation for smooth unknots. Given any non-zero rational surgery on an unknot, one may use Rolfsen twists to find a diffeomorphic manifold described by a rational surgery on the unknot with a surgery coefficient less than or equal to $-1$. So to describe the cosmetic surgeries on the unknot, we start with a rational number $-p/q\leq -1$. Now define 
\[
CS(-p/q)=\{ -p/q' | q'=q+np \text{ or } q'=\overline q +np \text{ for some } n\in\Z\}
\]
where $\overline q$ is the inverse of $q$ mod $p$ if it exists (otherwise, we ignore the second possibility for $q'$). Since any surgery on the unknot yields a lens space and we know when two lens spaces are diffeomorphic, it is easy to see that given $-p/q\leq -1$, then $r$ surgery on the unknot is orientation preserving diffeomorphic to $-p/q$ surgery on the unknot if and only if $r\in CS(-p/q)$. 
\begin{theorem}\label{unknotsurg}
For a Legendrian unknot $L$ with $|\rot(L)|<|\tb(L)+1|$ contact $(r-\tb(L))$-surgery on $L$ is not equivalent to any other contact surgery on $L$ if $r<\tb(L)$. For any other Legendrian unknot or non-zero contact surgery corresponding to a non-zero smooth surgery, there are infinitely many other contact surgeries on the unknot yielding contactomorphic manifolds. 

Specifically, let L be a Legendrian unknot with $|\rot(L)|= |\tb(L) + 1|$, and
let $-p/q <-1$. 
Suppose $r,r' \in CS(-p/q)$, with $r,r'$ either both
in the interval $(\tb(L),0)$ both less than $\tb(L)$, or both greater than $0$.
Then any contact $(r-\tb(L))$-surgery on L is contactomorphic to
some $(r'-tb(L))$ surgery on $L$. Moreover, any contact $(r-\tb(L))$-surgery on $L$ for $r<\tb(L)$ is equivalent to some contact $(r'-\tb(L))$-surgery for $r'>0$. 



Let $L$ be a Legendrian unknot with $|\rot(L)|<|\tb(L)+1|$ and $r, r'\in CS(-p/q)$ be greater than $\tb(L)$. Then any contact $(r-\tb(L))$-surgery on $L$ is contactomorphic to a contact $(r'-\tb(L))$-surgery on $L$.
\end{theorem}
\begin{remark}
We note the interesting phenomenon that any smooth surgery on the unknot (other than $0$ surgery) is diffeomorphic to infinitely many other surgeries on the unknot, but for Legendrian knots, there are some that have unique contact surgeries, though most have infinitely many cosmetic contact surgeries. 
\end{remark}

\begin{remark}
We will see in the proof of this theorem that for $r\in (0,1)$, there can be many different contact $(r-\tb(L))$-surgeries on a Legendrian unknot $L$ which give the same contact manifold. Thus, we see that there are contact surgeries on unknots that are strongly-cosmetic. 
\end{remark}

\begin{remark}
You can think of this theorem as giving a version of a contact Rolfsen twist. Specifically for the maximal Thurston-Bennequin unknot one may always perform positive and negative Rolfsen twists as long as both contact structures are tight. If the Thurston-Bennequin of the Legendrian unknot is not maximal, then one may still perform Rolfsen twist if the rotation number is extremal, but only negative Rolfsen twists are guaranteed, while positive ones might or might not be possible. 

We will give a diagrammatic interpretation of part of the theorem in Section~\ref{unknot} (the parts that involve tight contact structures) and see the full interpretation of contact Rolfsen twist in this context. 
\end{remark}

This theorem will be proven using a careful analysis of contact surgery and the classification of tight contact structures on various simple $3$-manifolds. The details can be found in Section~\ref{unknot}. 

\begin{remark}
Previously, Chatterjee and Kegel \cite{ChatterjeeKegel24pre} had studied surgeries on Legendrian unknots in the tight contact structure on $S^3$. Specifically, in Section~3 of their paper, they characterized which contact surgeries were tight and which were overtwisted (but did not specifically identify the contact structures obtained), then in Theorem~4.1 they did identify the overtwisted contact structures obtained by some contact surgery corresponding to $-(4m+3)/4$ surgeries on Legendrian unknots. Their proofs involved contact surgery diagrams and computing $d_3$-invariants of the specific surgeries considered. While it is possible that similar techniques could be used to prove our theorem above, it is not clear how this could be done, given that the methods we have for computing $d_3$-invariants for arbitrary contact surgeries on the unknot do not have simple closed-form expressions and can get quite complicated for surgery coefficients with long continued fraction expansions. So we will base our proof on the Farey graph and simple classification results on tori and lens spaces. 
\end{remark}

\subsection*{Acknowledgements}
The authors thank B\"ulent Tosun for helpful conversations about his work with Tom Mark. We also thank the referee for many valuable comments that improved the paper, especially for suggesting the exploration of the different types of contact cosmetic surgeries and the diagrammatic approach to the contact Rolfsen twist. The first author was partially supported by National Science Foundation grant DMS-2203312 and the Georgia Institute of Technology Elaine M. Hubbard Distinguished Faculty Award. Some of this work was completed while the second author was an ESI Fellow in Vienna, Austria. Currently, the second author receives partial support from the Infosys Fellowship.



\section{Background and preliminary result}\label{background}

We assume the reader is familiar with contact geometry and convex surface theory as discussed in \cite{Honda00a}. However, we recall some facts about the classification of contact structures in Section~\ref{class} for the convenience of the reader; more details on this discussion and the relevant notation can be found in \cite{EtnyreRoy21}. We then recall some facts about contact surgery on Legendrian knots in Section~\ref{surgsec}, while Section~\ref{d3set} discusses the $d_3$-invariant of plane fields. We end this section by recalling several results about the smooth cosmetic surgery conjecture. 

We also note here that we are using conventions for the Farey graph from \cite{EtnyreRoy21}. In particular, $0$ is at the point $(0,1)$ and $\infty$ is at the point $(0,-1)$ on the unit disk, and the positive rational numbers have positive $x$-coordinate. This differs from some early references in contact geometry, but fits better with topological slope conventions.

\subsection{Contact structures on thickened tori, solid tori, and lens spaces}\label{class}

A contact structure $\xi$ on $T^2\times [0,1]$ with convex boundary, where each boundary component has two dividing curves with slope $s_i$ on $T^2\times\{i\}$ is \dfn{minimally twisting} if any convex torus parallel to the boundary has dividing slope clockwise of $s_0$ and anti-clockwise of $s_1$ (we denote this by saying its slope is in $[s_0,s_1]$). From \cite{Giroux00, Honda00a} we know that any minimally twisting contact structure on $T^2\times [0,1]$ with boundary conditions as above is determined by a minimal path in the Farey graph from $s_0$ clockwise to $s_1$ with signs on each edge in the path. 

To discuss contact structures on solid tori, we set up some notation. Consider $T^2\times [0,1]$. If we foliate $T^2\times\{0\}$ by linear curves of slope $r$ and let $S_r$ be the result of collapsing the leaves of this foliation, then we call $S_r$ a \dfn{solid torus with lower meridian $r$}. Similarly, we can define the \dfn{solid torus $S^r$ with upper meridian $r$} by collapsing the same curves on $T^2\times\{1\}$. We note that the standard way of thinking of the solid torus as $S^1\times D^2$ is $S_\infty$ in this notation. Moreover, if one performs $r$ Dehn surgery on a knot $K$ in a manifold $M$ then this is equivalent to removing a standard neighborhood of $K$ and replacing it with $S_r$. 

From \cite{Giroux00, Honda00a}, we know that any tight contact structure on $S_r$ with convex boundary having dividing slope $s$ is determined by a minimal path in the Farey graph from $r$ clockwise to $s$ with a sign on each edge except the first. (We have a similar description for $S^r$ with dividing slope $s$ except the path runs from $r$ anti-clockwise to $s$.) Moreover, any such path gives a tight contact structure on $S_r$. 

Moreover, as shown in \cite{Honda00a}, if we glue two contact structures on thickened tori determined by minimal signed paths in the Farey graph, the result will be tight if when shortening the concatenated paths to a shortest path, the signs of any two shortened edges are the same. The same holds when gluing a thickened torus to a solid torus, except that the contact structure remains tight when shortening the path adjacent to the unlabeled edge.

Now one can describe a lens space by taking $T^2\times [0,1]$ and collapsing a linear foliation of slope $s$ on $T\times \{0\}$ and of slope $r$ on $T^2\times\{1\}$, denote this by $L_s^r$. The standard lens space $L(p,q)$ is $L^0_{-p/q}$ where $-p/q<-1$. Tight contact structures on $L_s^r$ are determined by a minimal path in the Farey graph from $s$ clockwise to $r$ with signs on all the edges except the first and the last, \cite{Giroux00, Honda00a}.

\subsection{Contact surgery}\label{surgsec}
Given a Legendrian knot $L$ in a contact manifold $(M,\xi)$ it has a standard neighborhood $N_L$. The boundary of $N_L$ is convex with dividing curves of slope $\tb(L)$. If $s$ is any rational number not equal to $0$, then contact $(s)$-surgery is defined to be the result of removing $N_L$ from $M$ and replacing it with a tight contact structure on the solid torus $S_{s+\tb(L)}$ with lower meridional $s+\tb(L)$. We know from the previous section that this contact structure is defined by a minimal path in the Farey graph from $s+\tb(L)$ clockwise to $\tb(L)$ with signs on all but the first edge in the path. So ``contact surgery" is not uniquely defined unless there is an edge from $s$ to $\tb(L)$, and this will occur when doing contact $(\pm 1/n)$-surgery on $L$. 

We note that above we did not define contact $(0)$-surgery. This is because contact surgery should remove a neighborhood of a Legendrian knot and replace it with another solid torus supporting a tight contact structure. But for ``contact $(0)$-surgery'', we must glue in an overtwisted contact torus since a Legendrian divide on the boundary of the torus will bound a disk and thus the contact structure must be overtwisted. If one allows overtwisted contact structures on the surgery torus, then there are always infinitely many contact surgeries, and the obviously overtwisted ones are understood by the classification of overtwisted contact structures. Thus, in this paper, we do not consider contact $(0)$-surgery to be well-defined and only consider contact surgeries with other coefficients. 

There is an algorithm for turning contact $(s)$-surgery on $L$ into a sequence of contact $(\pm 1)$-surgeries on some link $L'$ obtained from $L$. This was first described in \cite{DingGeigesStipsicz04}. We describe a modified version of this here (see \cite{Schoenenberger05} for the connection). When $s<0$ then let $s=[c_1,c_2,\ldots, c_n]$ be the continued fraction expansion. Let $L'$ be the link obtained from $L$ by adding a chain of $(n-1)$ unknots linked to $L$ (that is, $L_1$ is $L$, $L_2$ is a meridian to $L_1$ and $L_3$ is a meridian of $L_2$ and so on) and then stabilizing the first component $|c_1+1|$ times and $i^{th}$ component, for $i>1$, $|c_i+2|$ times. Contact $(s)$-surgery on $L$ is equivalent to contact $(-1)$-surgery on $L'$. If $r>0$ then let $p/q=s$. There is a smallest positive integer $k$ such that $p/(q-kp)$ is negative. Then let $L'$ be the Legendrian link obtained for $L$ by taking $k+1$ Legendrian push-offs of $L$ and contact $(s)$-surgery on $L$ is the same as contact $(+1)$-surgery on $k$ components of $L'$ and contact $(p/(q-kp))$-surgery on the $(k+1)^{st}$ component of $L'$. 

\subsection{The $d_3$-invariant}\label{d3set}
In \cite{Gompf98}, Gompf defined an invariant of homotopy classes of plane fields that obstructed them from being homotopic over the $3$-skeleton of a $ 3$-manifold. This invariant goes by many names (differing by multiplicative or additive constants), but we will use the following normalization. Given a plane field $\xi$ on a $3$-manifold $M$ one can find a $4$-manifold $X$ with an almost complex structure $J$ such that $\partial X=M$ and $\xi$ is the $J$-tangencies to the boundary (that is $\xi=TM\cap JTM$), see \cite{Gompf98}. We define the $d_3$-invariant of $\xi$ to be 
\[
d_3(\xi)=\frac 14(c_1^2(J)-3\sigma(X)-2(\chi(X)-1))\footnote{We note that we use $\chi(X)-1$ in the formula instead of $\chi(X)$. This is because with this definition, the invariant is additive under connected sum and has nice relations with other invariants.},
\] 
where $\sigma(X)$ is the signature of $X$, $\chi(X)$ is the Euler characteristic of $X$, and $c_1^2(J)$ is the first Chern class of $TM$ with the almost complex structure $J$. To make sense of $c_1^2(J)$ we need to assume that the Euler class (or Chern class) of $\xi$ is torsion. In that case, if one lets $Q$ be the intersection matrix of $X$, then there is a class $c\in H_2(X;\Q)$ such that $Q c$ is the Poincar\'e dual of $c_1(J)$. Then $c_1^2(J)$ is the intersection of $c$ with itself. That is, if one chooses a basis for $H_2(X;\Q)$, then $Q$ and $c$ are represented by a matrix and a vector, and $c_1^2(J)=c^T Q c$ where $c^T$ is the transpose of $c$. 

Given a Legendrian link $L$ with components $L_1\cup\cdots \cup  L_k$ and the contact structure $\xi$ is obtained from the standard contact structure on $S^3$ by contact $(+1)$-surgery on the $L_i$ for $i=1\ldots, l$ and contact $(-1)$-surgery on the $L_i$ for $i=l+1,\ldots, k$, then the $d_3$ invariant can be computed by 
\[
d_3(\xi)=\frac14 (c^2-3\sigma(X)-2(\chi(X)-1))+l
\]
where $X$ is the $4$-manifold obtained by attaching $2$-handles to $B^4$ along the link $L$ with framing $\tb(L_i)+1$ for $i=1,\ldots, l$ and framing $\tb(L_i)-1$ for $i=l+1,\ldots, k$, and $c^2$ is computed as follows. Let $Q$ be the intersection form of $X$ and let $\mathbf{r}$ be the column vector with $i^{th}$ entry the rotation number of $L_i$. Now set $\mathbf{r}'=Q^{-1} \mathbf{r}$ and $c^2=(\mathbf{r}')^TQ{\mathbf{r}}'={\mathbf{r}'}^T\mathbf{r}$. 

\subsection{Past results on the smooth cosmetic surgery conjecture}
Recently, Hanselman used Heegaard Floer invariants to prove the cosmetic surgery conjecture for most classes of knots and surgery coefficients. 
 \begin{theorem} [Hanselman, 2023 \cite{Hanselman2023}]\label{hanselman}
 	If $K$ is a non-trivial knot in $S^3$ and $S^3_{K}(r) \cong S^3_K(r')$, then we have the following:
 	\begin{itemize}
 		\item The pair of slopes {$r,r'$} are either {$\pm2$} or {$\pm \frac{1}{n}$} for some positive integer $n$;
 		\item if {$r,r'$} are {$\pm2$} then $g(K)=2$
 	\end{itemize}
 \end{theorem}
 
 Even more recently, we have further restrictions on cosmetic surgeries.
 \begin{theorem}[Daemi, Lidman, and Miller Eismeier, 2024 \cite{DaemiEismeierLidman24pre}]
 For a non-trivial knot $K$ in $S^3$, $1/n$ surgery on $K$ is not orientation preserving diffeomorphic to $-1/n$ surgery for any $n\not=0$. 
 \end{theorem}
 
 The two theorems above imply that one only needs to check if $2$ and $-2$ surgery on a knot results in diffeomorphic manifolds. 
 
Plamenevskaya showed that the Ozsv\'{a}th-Szab\'{o} concordance invariant for a knot, $\tau$, gives a bound on classical invariants for Legendrian knots and Ni and Wu gave a restriction on this invariant for cosmetic surgery on the knot. We state these results here.
\begin{theorem}[Plamenevskaya 2004, \cite{Plamenevskaya2004}]\label{taubound}
	For a Legendrian knot $L$ in $(S^3,\xi_{std})$,
	$$\tb(L)+|\rot(L)|\leq 2\tau(L)-1.$$ 
\end{theorem}

\begin{theorem}[Ni and Wu, 2013\cite{NiWu2015}]\label{tauresult} 
Suppose $S^3_{K}(\frac{p}{q}) \cong S^3_K(\frac{p}{q'})$ with $q\neq q'.$ Then $\tau (K)=0$, where $\tau$ is the Ozsv\'{a}th-Szab\'{o} concordance invariant.
\end{theorem}

	\section{Cosmetic contact surgeries on unknot}\label{unknot}
	
	To relate surgeries on the unknot, we recall the Rolfsen twist. Specifically, $p/q$ surgery on an unknot is equivalent to $p/(q+np)$ surgery on the unknot. We will need to know why this is true, so sketch the proof now. If $U$ is a smooth unknot in $S^3$, we let $N_U$ be the solid torus neighborhood of $U$ and $S^3_U$ its complement in $S^3$. Now $S^3_U(p/q)$ is the result of removing $N_U$ from $S^3$ and regluing it to $S^3_U$ by the map $f: \partial N_U\to -\partial S^3_U$ given by the matrix
	\[
	f=\begin{bmatrix}
		q'&q\\p'&p
	\end{bmatrix}, 
	\]
where we are using the standard longitude-meridian basis on $\partial N_U=\partial S^3_U$ to express the map as a matrix. 
	Now consider the diffeomorphism $\phi_n:S^3_U\to S^3_U: (\phi,(r,\theta))\mapsto (\phi,(r,\theta+n\phi))$. Here we are identifying $S^3_U$ with $S^1\times D^2$. 
	We can build a diffeomorphism from $S^3_U(p/q)=S^3_U\cup_f N_U$ to $S^3_U\cup _{\phi\circ f} N_U$ by sending $S^3_U$ to $S^3_U$ by $\phi_n$ and $N_U$ to $N_U$ by the identity map. The latter manifold is $S^3_U(p/(p+nq))$, thus establishing the diffeomorphism claimed by the Rolfsen twist. 
	
	Another useful way to see this diffeomorphism is via upper and lower meridians, as discussed in the previous section. That is, we can take $T^2\times [0,1]$ and then we obtain $S^3_U(p/q)$ by taking a foliation of $T^1\times \{1\}$ by linear curves of slope $0$ and collapse them to obtain a solid torus with upper meridian $0$. Now if we take a foliation of $T^2\times\{0\}$ by linear curves of slope $p/q$ and collapse them, we get $S^3_U(p/q)$. We see that $T^2\times\{1/2\}$ splits this manifold into two solid tori, one with upper meridian $0$ (that will correspond to $S^3_U$) and one with lower meridian $p/q$ (that will correspond to the reglued $N_U$). Now we can see the Rolfsen twist as follows: Apply the map $\phi_n$ to $T^2\times [0,1]$. Notice that this does not affect the $0$ sloped curves on $T^2\times\{1\}$ and send the $p/q$ sloped curves on $T^2\times\{0\}$ to curves of slope $p/(q+np)$. Thus, after collapsing the leaves of these foliations, we have a diffeomorphism between $S^3_U(p/q)$ and $S^3_U(p/(q+np))$. 
	
	One may easily check that given any surgery on $U$, one may perform Rolfsen twists to get an equivalent surgery with surgery coefficient $-p/q<-1$, and such a surgery coefficient is unique. Now, the first family of equivalent surgeries in the set $CS(-p/q)$ defined in the introduction are all obtained via Rolfsen twists. 
	
	To understand the second family of equivalent surgeries in the set $CS(-p/q)$, we note that one can perform (inverse) slam dunk moves to write $-p/q$ surgery on the unknot as surgery on a chain of $k$ unknots with surgery coefficients $-a_1,-a_2,\ldots -a_k$ where $[-a_1,\ldots, -a_k]$ is the continued fraction of $-p/q<-1$. Now performing slam dunks in the reverse order, one obtains surgery on the unknot with surgery coefficient $[-a_k,\ldots, -a_1]$ and one may check by induction on $k$ that this is $-p/\overline q$ surgery where $\overline q$ is the inverse of $q$ mod $p$. 
	
\begin{proof}[Proof of Theorem~\ref{unknotsurg}]
Suppose that $L$ is a Legendrian unknot with $|\rot(L)|<|\tb(L)+1|$. Let $N_L$ be a standard neighborhood of $L$ and $S^3_L$ its complement. So $N_L$ is a solid torus with lower meridian $\infty$ and convex boundary of slope $\tb(L)$ and $S^3_L$ is a solid torus with upper meridian $0$ and convex boundary with dividing slope $\tb(L)$ (we note that all slopes are measured with respect to longitude-meridian coordinates given by $L$). Now if $r<\tb(L)$, then any contact $(r-\tb(L))$-surgery on $L$ can be achieved by a sequence of Legendrian surgeries and is hence tight. On the other hand, we claim any other contact $(r'-\tb(L))$-surgery for $r'\in CS(-p/q)$  will be overtwisted, thus completing the first claim of the theorem. To see this, we note the contact structure on $S^3_L$ is determined by signs on the path in the Farey graph from $\tb(L)$ clockwise to $-1$, and because of the condition on $\rot(L)$, the signs cannot all be the same. (To see this, we note that the condition on $\rot(L)$ says that $L$ has been stabilized positively and negatively and thus, positive and negative basic slices are added to its complement \cite{EtnyreHonda01b}.) Now if $r'\in CS(-p/q)$ is not less than $\tb(L)$ then contact $(r'-\tb(L))$-surgery on $L$ is the result of removing $N_L$ from $S^3$ and gluing in a solid torus $S_{r'}$ with a tight contact structure determined by a minimal path in the Farey graph from $r'$ to $\tb(L)$ with signs on all but the first edge. Note that one edge in this path goes from $\infty$ to $\tb(L)$. Taking the corresponding basic slice in $S_{r'}$ and gluing it to $S^3_L$ we see that there will be a contact structure on $T^2\times[0,1]$ described by a path in the Farey graph from $\infty$ to $-1$, but with edges from $\infty$ to $\tb(L)$ then $\tb(L)+1$ and so on until we arrive at $-1$. Since this path can be shortened to a path with just one edge from $\infty$ to $-1$ and the signs on the edges are not all the same, we know that the contact structure is overtwisted. 
		
		We now turn to the case of a Legendrian knot $L$ with $\rot(L)=\pm |\tb(L)+1|$ (this means that $L$ has been only stabilized positively or only stabilized negatively). We will discuss the case when $\tb(L)=-1$ and discuss the other cases later. Now, if $N_L$ is a standard neighborhood of $L$ and $S^3_L$ is its complement, then $N_L$ is a solid torus with lower meridian $\infty$ and convex boundary with dividing slope $-1$, and $S^3_L$ will be a solid torus with upper meridian $0$ and convex boundary with dividing slope $-1$. Now, contact $(r+1)$ surgery on $L$ is obtained by removing $N_L$ and gluing in a solid torus $S_r$ with lower meridian $r$ and convex boundary with dividing slope $-1$. The contact surgery is described by a minimal path in the Farey graph from $r$ clockwise to $-1$ with decorations on all but the first edge in the path. 
		
		We now note the difference between $r\in(-1,0)$ and $r\not\in(-1,0)$. When $r\in(-1,0)$ then $S_r$ has a tight contact structure with dividing slope $-1$ and lower meridian $r$. So we know we can find a convex torus $T$ in $S_r$ with any dividing slope in $(r,-1]$ (recall this notation means any rational number in the Farey graph that is clockwise of $r$ and anti-clockwise of $-1$), thus there is a convex torus $T$ with dividing slope $0$. Now, a Legendrian divide on $T$ will bound a disk outside the solid torus $T$ bounds in $S_r$ and thus, the contact structure will be overtwisted. However, if $r\not\in(-1,0)$ then some contact $(r+1)$-surgery on $L$ will be tight (since it will correspond to a minimal decorated path in the Farey graph). Thus, we need to consider the cases separately. 
		
	We begin with $r\not\in(-1,0)$. Any such $r$ is related through Rolfsen twists to a rational number less than $-1$, so we will assume that $r<-1$. Below, we will show that for any $r'\in CS(r)$ with $r'\not\in (-1,0)$, any contact $(r+1)$-surgery on $L$ is equivalent to some contact $(r'+1)$-surgery surgery on $L$, and vice versa. 
		
		Using our description of surgery using upper and lower meridians recalled just before this proof we see that $S_r$ is just $T^2\times[0,1/2]$ with curves of slope $r$ collapsed on $T^2\times \{0\}$ and $S^3_L$ is $T^2\times[1/2,1]$ with curves of slope $0$ collapsed on $T^2\times\{1\}$. See Figure~\ref{surgex}.
		\begin{figure}[tb]{\footnotesize
				\begin{overpic}
					{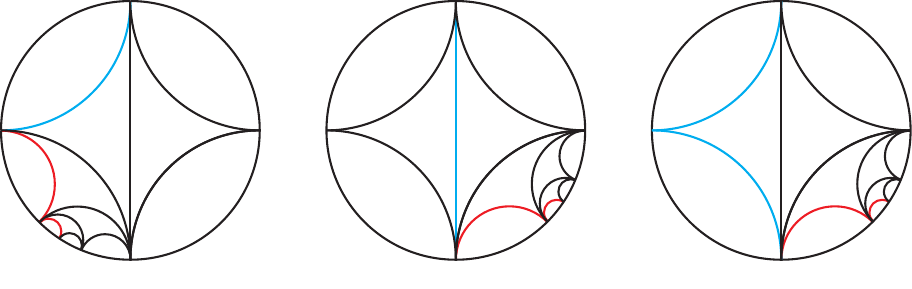}
					\put(61, 142){$0$}
					\put(-12, 74){$-1$}
					\put(6, 25){$-2$}
					\put(29, 10){$-3$}
					\put(15, 15){$-\frac 52$}
					\put(28, 55){$+$}
					\put(128, 74){$1$}
					\put(59, 5){$\infty$}
					\put(218, 142){$0$}
					\put(142, 74){$-1$}
					\put(285, 74){$1$}
					\put(264, 26){$2$}
					\put(280, 49){$\frac 32$}
					\put(274, 38){$\frac 53$}
					\put(233, 42){$+$}
					\put(215, 5){$\infty$}
					\put(374, 142){$0$}
					\put(300, 74){$-1$}
					\put(440, 74){$1$}
					\put(423, 26){$2$}
					\put(438, 49){$\frac 32$}
					\put(432, 38){$\frac 53$}
					\put(392, 42){$+$}
					\put(360, 60){$\pm$}
					\put(372, 5){$\infty$}
			\end{overpic}}
			\caption{On the left we see the result of contact $(-3/2)$-surgery on $L$. The blue path describes the contact structure on $S^3_L$ and the red path describes the contact structure on $S_r$. In the middle figure, we see the same manifold after applying the coordinate change $\phi_1$. On the right, we see that the image of the $S^3_L$ can be split into a solid torus (with upper meridian $0$ and dividing slope $-1)$ and a thickened torus (with dividing slopes $-1$ and $\infty$). Attaching the thickened torus to the image of $S_r$ under $\phi_1$ shows that this manifold is the result of contact $(8/3)$-surgery on $L$, and different such surgeries give this manifold, as there are two choices for the sign describing the contact structure on the thickened torus.}
			\label{surgex}
		\end{figure}
		Now if $r'$ is obtained from $r$ by a Rolfsen twist given by $\phi_n$ and $r'\not\in(-1,0)$ (notice this, and $r<-1$, implies that $n\geq 0$), then we simply apply $\phi_n$ to $T^2\times[0,1]$ and collapse curves of slope $0$ on $T^2\times\{1\}$ and curves of slope $r'$ on $T^2\times \{0\}$. Now the torus $\partial N_L=T^2\times\{1/2\}$ will map to a convex torus with dividing slope $1/(n-1)$. So under this diffeomorphism, the torus $S_r$ with dividing slope $-1$ and meridional slope $r$ will become a solid torus $S_{r'}$ with dividing slope $1/(n-1)$ and meridional slope $r'$. Its complement, which we call $C_L$, will be a solid torus with upper meridian $0$ and dividing slope $1/(n-1)$. We note that the path in the Farey graph describing the contact structure on $S_r$ becomes, under the diffeomorphism, a path describing a unique contact structure on $S_{r'}$. Moreover, since $1/(n-1)$ and $0$ share an edge in the Farey graph, we know there is a unique tight contact structure on $C_L$, and it is a neighborhood of a Legendrian knot $L'$. Inside of $C_L$ there is a convex torus $T$ with dividing slope $-1$ (since $-1\in [1/(n-1),0)$). The torus $T$ will split $C_L$ into two pieces. One piece is simply $S^3_L$, and the other piece is a thickened torus $T\times I$ with dividing curves of slope $-1$ and $1/(n-1)$. Notice that there is a path from $1/(n-1)$ to $-1$ in the Farey graph with $n$ edges, and this path forms a continued fraction block. So there are $n+1$ possible contact structures on $T^2\times I$, but all of them when glued to $S^3_L$ become the same since there is a unique tight contact structure on $C_L$ (note $S^3_L$ can be thought of as a neighborhood of an $ n$-fold stabilization of $L'$, the Legendrian knot determined by $C_L$). Thus,  there are $n+1$ tight contact structures on the complement of $S^3_L$ that will give contact structures on $S^3_L(r')$ that are contactomorphic to the one on $S^3_L(r)$. 
		
	\begin{remark}
		We note that for $r'>0$, there can be many different contact surgeries that give the same contact structure on the lens space. So there are many contact cosmetic surgeries even for the same rational number $r'$. 
	\end{remark}
		
		Now, if $r'\in CS(-p/q)$ is not obtained from $r$ by a Rolfsen twist, then it is related to $r$ by slam dunks, as discussed just before this proof. (Technically, one might have to perform Rolfsen twists on $r$ to ensure it is less than $-1$, then perform the slam dunk procedure, and then perform more Rolfsen twists to get $r'$, but since we already understand the Rolfsen twists' effect on contact surgery, we can ignore this issue.) Since we know all the contact structures on a lens space come from contact surgery on $L$, we know that under the diffeomorphism from $L(p,q)$ to $L(p,\overline q)$ any contact structure on $L(p,q)$ comes from some contact $(-p/\overline q +1)$ surgery on $L$. This completes the argument for surgery coefficients $r\not\in (-1,0)$. 
		
		Now, when $r\in(-1,0)$, we can give the same argument to see there are equivalent contact surgeries for other $r'\in CS(-p/q)$ in $(-1,0)$, just in this case, as noted above, the contact structures will be overtwisted. 
		
		We now turn to the case of a Legendrian knot $L$ with $|\rot(L)|=|\tb(L)+1|$, but $\tb(L)<-1$. Here we will see that we need to consider the cases when $r\not\in(\tb(L),0)$ and when $r\in (\tb(L),0)$. Arguing as in the case when $\tb(L)=-1$, we see that in the latter case the contact structures obtained from contact surgery will be overtwisted, and in the former case they will be tight. We consider the case when $r\not\in(\tb(L),0)$ and, as in the argument above, we start with $r<\tb(L)$. Now, for any other $r'\in CS(-p/q)$ with $r'\not\in(\tb(L),0)$, we can argue as we did in the $\tb(L)=-1$ case, but there is one extra thing to consider. Recall, in the above argument, we split $C_L$ into two pieces, one being $S^3_L$ and the other being a thickened torus $T^2\times I$. The dividing slope on $S^3_L$ is $\tb(L)$ and the contact structure is determined by a path in the Farey graph from $\tb(L)$ clockwise to $0$ with all edges (but the last, which does not have a sign) being decorated with all $+$ or all $-$ (this is because $|\rot(L)|=|\tb(L)+1|$). The thickened torus will have dividing curves of slope $1/(n-1)$ and $\tb(L)$. Notice that the minimal path from $1/(n-1)$ to $\tb(L)$ goes through $\infty$. If the sign of the basic slice with slopes $\infty$ and $\tb(L)$ is different from the signs describing the contact structure on $S^3_L$, then the union of this basic slice and $S^3_L$ will be overtwisted; otherwise, it will be tight. When it is tight, the same argument as in the $\tb(L)=-1$ case will show that every contact $(r-\tb(L))$-surgery on $L$ will be equivalent to some contact $(r'-\tb(L))$-surgery on $L$. Since $r<\tb(L)$ we know that all contact $(r-\tb(L))$-surgery on $L$ are tight, thus there are contact $(r'-\tb(L))$-surgeries for $r'>0$ that are not the same as a contact $(r-\tb(L))$-surgery on $L$ (since they will be overtwisted). But by considering paths in the Farey graph and the effects of Rolfsen twists on them, as we did in the $\tb(L)=-1$ case above, we can see that any overtwisted contact structures obtained by contact $(r'-\tb(L))$-surgeries on $L$ for $r'>0$ is equivalent to contact $(r''-\tb(L))$-surgery on $L$ for any $r''\in CS(-p/q)$ with $r''>0$. 
		
		We are finally left to consider the case of a Legendrian knot $L$ with $|\rot(L)|<|\tb(L)+1|$ and we are doing contact $(r-\tb(L))$-surgery for $r$ not less than $\tb(L)$. As noted at the start of this proof, all these contact structures are overtwisted, and arguments similar to those above will give contactomorphic contact $(r'-\tb(L))$-surgery on $L$ for any $r'\in CS(-p/q)$ with $r'$ not less than $\tb(L)$. 
	\end{proof}
	
We now give a diagrammatic way to see parts of the above proof. As in the proof, we begin with the Legendrian unknot $L$ with $\tb(L)=-1$. On the left of Figure~\ref{Rolf1} we see a contact $(-p/(q-np)+1)$-surgery on $L$ expressed as contact $(\pm 1)$-surgery and contact $(-p/q+1)$-surgery on a link for $-p/q<-1$ and $n\geq 0$. The rest of the figure shows a sequence of ``contact Kirby moves'', see \cite{CasalsEtnyreKegel2024}, that result in contact $(-p/q+1)$-surgery on $L$. This is a diagrammatic reinterpretation of the fourth and fifth paragraphs of the proof above. 
\begin{figure}[htb]{\small
\begin{overpic}
{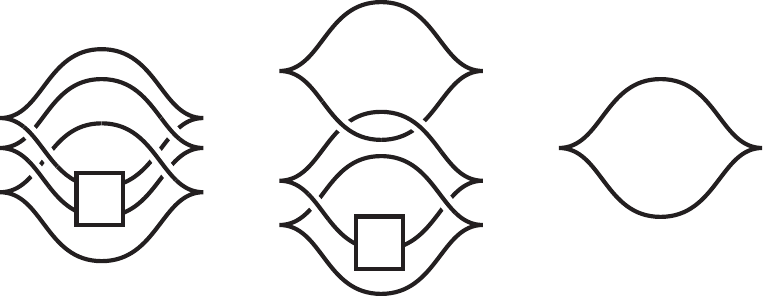}
\put(45, 44){$n$}
\put(180, 22){$n$}
\put(102, 47){$(+1)$}
\put(102, 68){$(-1)$}
\put(75, 110){$(-p/q+1)$}
\put(235, 31){$(+1)$}
\put(235, 54){$(-1)$}
\put(217, 127){$(-p/q+1)$}
\put(340, 100){$(-p/q+1)$}
\end{overpic}}
\caption{On the left, we see a surgery diagram that is equivalent to contact $(-p/(p-np)+1)$-surgery on the $L$. The box labeled $n$ indicates the $(-1)$-framed knot has been stabilized $n$ times (they could be any type of stabilization), and then the other knot going through the box is a Legendrian push-off of the $(-1)$-framed knot. In the middle diagram, we see the result of sliding the knot with surgery coefficient $(-p/q+1)$ over the one with coefficient $(-1)$. On the right, we see the result of canceling the components with coefficients $(+1)$ and $(-1)$.}
\label{Rolf1}
\end{figure}

We now consider the result of this contact Rolfsen twist on an entire surgery diagram. This is shown in Figure~\ref{Rolf2} and can be proven by going through the contact Kirby moves in Figure~\ref{Rolf1} followed by handle sliding the red curves over the curve with coefficient $(-1)$. 
\begin{figure}[htb]{\small
\begin{overpic}
{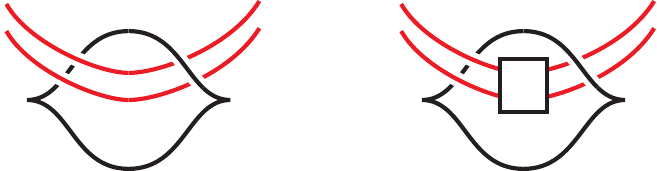}
\put(95, 10){$(-p/(q-np)+1)$}
\put(282, 10){$(-p/q+1)$}
\put(249, 39){$n$}
\end{overpic}}
\caption{The result of a contact Rolfsen twist. The red strands indicate any number of strands running through the unknot $L$, and the box labeled $n$ has the same interpretation as in the previous picture. }
\label{Rolf2}
\end{figure}
We note that a version of Figure~\ref{Rolf2} first appeared in \cite{CasalsEtnyreKegel2024} in the case when $-p/q=\infty$. 

Figure~\ref{Rolf2} shows how to take strands of a surgery diagram passing through $L$ with contact surgery coefficient $(-p/(q-np))$ and change to strands, with $n$ twists added, passing through $L$ with surgery coefficient $(-p/q+1)$. Of course, one may go in the other direction too, but one needs to have the strands passing through $L$ to already have the necessary twists. If one does not have such twists, one can still perform the contact Kirby moves to obtain a diagram for contact $(-p/(q-np))$-surgery on $L$, but the red strands will be tangled up in the surgery diagram describing the contact $(-p/(q-np))$-surgery. The proof above shows that one should be able to disentangle them, so that there is a simpler picture, but this seems difficult to do in practice. 

One can perform a similar analysis for contact surgeries on a Legendrian unknot $L$ with $\tb(L)=-n<-1$ if its rotation number is $\pm(n+1)$. The key to the analog of Figure~\ref{Rolf1} is to notice that a tight contact $(n)$-surgery on $L$ (and from the proof above, we know there is precisely one) gives the tight contact structure on $S^1\times S^2$, and so in a surgery diagram for contact $(-p/(q-np)+n)$-surgery on $L$, one can see contact $(n)$-surgery on $L$ and then replace it with contact $(+1)$-surgery on the maximal Thurston-Bennequin invariant unknot. From there one may proceed as above. 

\section{Limiting contact cosmetic surgeries}\label{obstruction}
In this section we prove Theorem~\ref{contactcosmeticthm}. It will be a direct consequence of Propositions~\ref{p1}, \ref{p2}, and~\ref{p3}, which deal with the case of Legendrian knots with $\tb=-1, -2,$ and less than $-2$, respectively, and the fact that Theorems~\ref{taubound} and~\ref{tauresult} show that the contact cosmetic surgery conjecture holds for Legendrian knots with $\tb\geq 0$.


\begin{proposition}\label{p1}
The contact cosmetic surgery conjecture holds for Legendrian knots with Thurston-Bennequin invariant $-1$ except possibly for $\pm 2$ surgery on a Legendrian knot $L$ in a knot type $K$ with $\tbb(K)=-1$, $\tau(K)=0$, $g(K)=2$, and $g_4(K)=0$. Moreover, $L$ is Lagrangian slice, and hence $K$ is quasi-positive. 
\end{proposition}
\begin{proof}
Let $L$ be a Legendrian knot with $\tb=-1$. Because we know that $\tau(L)$ must be zero if $L$ admits a cosmetic surgery by Theorem~\ref{tauresult} above, and we have the Bennequin type bound in Theorem~\ref{taubound} above, we see that the rotation number of $L$ must be $0$. 

We now consider $\pm 2$ surgeries on $L$. These are given by the surgery diagram in Figure~\ref{tbm1-2surg}. 
\begin{figure}[htb]{\small
\begin{overpic}
{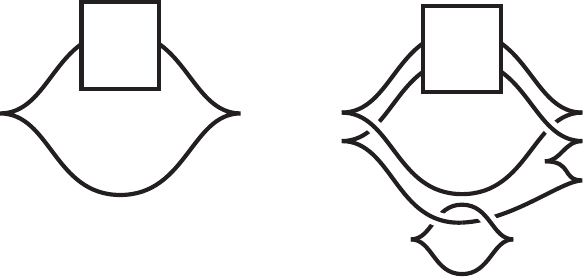}
\put(51, 106){\Large $K$}
\put(215, 106){\Large $K$}
\put(110, 90){$(-1)$}
\put(271, 90){$(+1)$}
\put(271, 30){$(-1)$}
\put(240, 2){$(-1)$}
\end{overpic}}
\caption{For a Legendrian knot $K$ with $\tb=-1$ we see a smooth $-2$ surgery on the left and a smooth $2$ surgery on the right (the stabilization can be either positive or negative).}
\label{tbm1-2surg}
\end{figure}
One may easily check that the $d_3$-invariant for both contact structures is $1/4$, so it is possible from this computation that contact cosmetic surgeries are the same. This is not surprising given that for the $\tb=-1$ unknot, these surgeries both produce the tight contact structure on $L(2,1)$ and hence are contactomorphic (and the $d_3$ computations for a general Legendrian knot with $\tb=-1$ and rotation number $0$ will yield the same result). So the best we can do in this situation is to limit the possible knots that might admit cosmetic surgeries. 

If $L$ is in the knot type 
$K$ then we know from Theorem~\ref{tauresult} above that $\tau(K)=0$ and from that and Theorem~\ref{taubound} (and our assumption that $\tb(L)=-1$) we know that $\tbb(K)=-1$. Moreover, $K$ has Seifert genus $2$ by Theorem~\ref{hanselman}, and since contact $(+3)$-surgery on $L$ is contactomorphic to contact $(-1)$-surgery on $L$ we know that it must be symplectically fillable and then \cite[Corollary~1.3]{MarkTosun2024} implies that $K$ is quasi-positive. In addition Proposition~1.7 in \cite{MarkTosun2024} says that if a positive contact surgery on $L$ is symplectically fillable, then the minimal possible smooth surgery coefficient that could be symplectically fillable is in $(2g_4(K),4g_4(K)]$ if $g_4(K)>0$, where $g_4(K)$ is the minimal genus of a surface in $B^4$ with boundary $K$. Since our $K$ has Seifert genus $2$, we know $g_4(K)$ is $0$, $1,$ or $2$. Since $2$ is not in $(2,4]$ or $(4, 8]$ we see that we must have $g_4(K)=0$. Now Corollary 2.10 in \cite{MarkTosun2024} implies that $L$ is Lagrangian slice.  
\end{proof}

We now turn to Legendrian knots with Thurston-Bennequin invariant $-2$.
\begin{proposition}\label{p2}
The contact cosmetic surgery conjecture holds for Legendrian knots with Thurston-Bennequin invariant $-2$. 
\end{proposition}
\begin{proof}
Let $L$ be a Legendrian knot with Thurston-Bennequin invariant $-2$. 
Recall that we only need to check whether smooth $\pm 2$ surgeries yield cosmetic surgeries. Note that smooth $-2$ surgery on $L$ is a contact $(0)$-surgery and so is not well-defined, and so it cannot be contactomorphic to $+2$ surgery (that is, contact $(4)$-surgery).
\end{proof}

We finally consider Legendrian knots with $\tb<-2$. 
\begin{proposition}\label{p3}
The contact cosmetic surgery conjecture holds for Legendrian knots with Thurston-Bennequin invariant less than $-2$. 
\end{proposition}
\begin{proof}
We consider smooth $\pm 2$ surgery on a Legendrian knot $L$ with $\tb(L)<-2$. The contact surgery diagrams for these surgeries are given in Figure~\ref{tbmk-2surg}. We denote the $4$-manifold in the diagram by $X_{\pm 2}$, where the Thurston-Bennequin invariant of the Legendrian knot will be clear from the context. 

\begin{figure}[htb]{\footnotesize
\begin{overpic}
{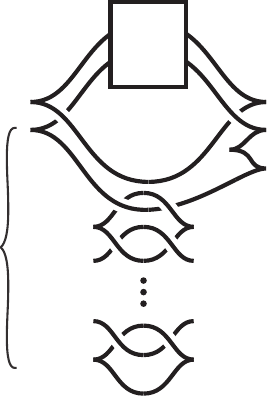}
\put(65, 165){\Large $K$}
\put(130, 142){$(+1)$}
\put(130, 118){$(-1)$}
\put(96, 82){$(-1)$}
\put(96, 18){$(-1)$}
\put(-8, 72){$a$}
\end{overpic}}
\caption{For a Legendrian knot $K$ with $\tb=-k<-2$ we see a smooth $-2$ surgery (that is contact $(k-2)$ surgery) when $a=k-3$ and smooth $2$ surgery (that is contact $(k+2)$ surgery) when $a=k+1$.}
\label{tbmk-2surg}
\end{figure}
It will be convenient to consider the $\tb=-3$ case first. 
Recall from the bound in Theorem~\ref{taubound} that the rotation number of the Legendrian knot $L$ must be $\pm 2$ or $0$.
In this case the intersection matrix of $X_{-2}$ is 
$$
\begin{bmatrix}
-2
\end{bmatrix}.
$$
From this, we see that $\chi(X_{-2})=2$ and $\sigma(X_{-2})=-1$. One may also compute $c_1^2(X_{-2})=\frac{-r^2}{2}$ and hence $c_1^2(X_{-2})= - 2$ or $0$. This gives
\[
d_3(\partial X_{-2})= 3/4 \text{ or } 5/4.
\]
 
 Now the intersection form of $X_2$ is
$$
\begin{bmatrix}
-2 & -3 &  & &   \\
-3 & -5 & -1 & &    \\
& -1 & -2& -1&    \\
  & & -1 & -2 & -1 \\
 &&&-1 & -2   \\
\end{bmatrix}.
$$
One may easily compute that $\sigma(X_2)=-3$ and $\chi(X_2)=6$. One may also compute that $c_1^2(X_2)=-2,4,$ or $14$ 
and hence
\[
d_3(\partial X_2)= 1/4, 7/4, \text{ or } 17/4.
\]
Thus, there is no contact cosmetic surgery in this case. 

We now consider $\pm 2$ surgery on a Legendrian knot $L$ with $\tb(L)=-k<-3$. We first note that Theorems~\ref{taubound} and~\ref{tauresult} imply that the rotation number of $L$ must lie within the range $k-1,k-3,\ldots, -k+1$.  
Let $X_{\pm 2}$ be the $4$-manifold in the surgery diagram for $\pm 2$ surgery on $L$ given in Figure~\ref{tbmk-2surg}. The intersection matrix for $X_{-2}$ is
$$
\begin{bmatrix}
-k+1 & -k &  & &   \\
-k & -k-2 & -1 & &   \\
 & -1 & -2 & \ddots&    \\
 &  & \ddots &\ddots & -1  \\
 &   &&-1 & -2   \\
\end{bmatrix}_{k-2 \times k-2}\\
$$
So we see that $\chi(X_{-2})=k-1$. Denote the rotation number of the link $L$ by $i$, for $i=k-1, k-3\ldots, -k+1$ and the rotation number of the stabilized knot in Figure~\ref{tbmk-2surg} is $i\pm 1$ and the other knots in the surgery diagram have rotation number $0$.  In Appendix~\ref{lac} we will prove the following lemma.
\begin{lemma}\label{tbleqm3andm2surgery}
The signature of $X_{-2}$ is $\sigma(X_{-2})=-k+2$ and 
\[
c_1^2(X_{-2})= -\frac 12 i^2 \pm (k-3) i + \frac 12 (-k^2+4k-3),
\]
where $i=k-1,k-3,\ldots, -k+1$. 
\end{lemma}
From this, we can compute 
\[
d_3(\partial X_{-2})=\frac 14\left(-\frac 12 i^2\pm (k-3)i+\frac 12(-k^2+6k-7)\right)+1
\]

We now turn to $2$ surgery on $L$.  The intersection matrix for $X_{2}$ is
$$
\begin{bmatrix}
	-k+1 & -k &  & &   \\
	-k & -k-2 & -1 & &   \\
	& -1 & -2 & \ddots&    \\
	&  & \ddots &\ddots & -1  \\
	&   &&-1 & -2   \\
\end{bmatrix}_{k+2 \times k+2}\\
$$
So we see that $\chi(X_{2})=k+3$. Denote the rotation number of the link $L$ by $i$, for $i=-k+1,\ldots, k-1$, and the rotation number of the stabilized knot in Figure~\ref{tbmk-2surg} is $i\pm 1$, and the other knots in the surgery diagram have rotation number $0$.  In Appendix~\ref{lac} we will prove the following lemma.
\begin{lemma}\label{tbleqm3and2surgery}
	The signature of $X_{2}$ is $\sigma(X_{2})=-k$ and 
	\[
	c_1^2(X_{2})= \frac12i^2\pm(k+1)i+\frac 12(k^2-1),
	\]
	where $i=k-1,k-3,\ldots, -k+1$. 
\end{lemma}
From this, we can compute
\[
d_3(\partial X_2) = \frac 14\left(\frac 12 i^2\pm (k+1)i +\frac 12(k^2+2k-9)\right) +1
\]
Using the quadratic equation (or Mathematica) to solve $d_3(\partial X_{-2})=d_3(\partial X_2)$ for $i$ yields no integer solutions. So there are no contact cosmetic surgeries with smooth surgery coefficients $\pm 2$ when $\tb(L)<-3$. 
\end{proof}

\section{The weakly-cosmetic surgery conjecture}
In this section, we will establish Proposition~\ref{weak} that gives conditions under which there are no contact weakly-cosmetic and strongly-cosmetic surgeries.
\begin{proof}[Proof of Proposition~\ref{weak}]
Item~(1) in the proposition says that there are no strongly-cosmetic surgeries on $L$ with smooth surgery coefficient $r<\tb(L)$. This easily follows from the fact that any such contact surgery will be obtained from Legendrian surgery on a link obtained from $L$ via some Legendrian push-offs and stabilizations. The different contact surgeries correspond to different stabilizations and hence different rotation numbers on the components of the Legendrian link. Thus, by \cite[Proposition~2.3]{Gompf98} we know that the Stein manifold that each of the contact manifolds bounds has a different Chern class. Theorem~1.2 of Lisca and Mati\'c \cite{LiscaMatic97} now says that all the contact structures constructed this way are distinct. 

We recall that Item~(2) in the proposition says that if contact $(+1/n)$-surgery for some $n>1$, respectively $n=1$, on $L$ has a non-vanishing contact invariant, then there are no strongly-cosmetic surgeries on $L$ for contact $(r)$-surgeries with $r>0$, respectively $r\geq 1$. We start with the case that $n>1$. We first note that by \cite[Theorem~1.2]{MarkTosun2018} we know that if contact $(1/n)$-surgery has non-vanishing contact invariant for some positive $n$, then so does contact $(1/m)$-surgery for all positive $m$. (Not also that such surgeries are unique.) Now, the key observation here is that any contact $(r)$-surgery for $r>0$ comes from Legendrian surgery on some link in contact $(1/m)$-surgery for some $m$. Now a cobordism version of the Lisca and Mati\'c theorem above from \cite{Simone2018pre}, tells us that all the different contact $(r)$-surgeries must be distinct. The same argument holds for $n=1$ except we must have the contact surgery coefficient is greater than or equal to $1$. 

Turning to Item~(3) in the proposition, we recall that it says that if $L$ is an L-space knot with $\tb(L)=2\tau(L)-1$, then $L$ admits no strongly-cosmetic surgeries. We notice that there are no strongly-cosmetic surgeries for any smooth surgery coefficient $r<\tb(L)$ by item~(1), and \cite[Theorem~1.2]{MarkTosun2018} implies that the hypothesis of Item~(2) is satisfied, so there are no strongly-cosmetic surgeries for any smooth surgery coefficient $r>\tb(L)$. (We recall that contact surgery is not defined for a smooth surgery coefficient equal to $\tb(L)$.)

We finally consider Item~(4), which says that for smooth surgery coefficient $-2$ a Legendrian knot $L$ does not have weakly-cosmetic surgeries if $\tb(L)\leq -2$ except possibly when $\rot(L)=0$, and no strongly-cosmetic surgeries, except possibly when $\rot(L)=0$. 
We begin by considering a Legendrian knot $L$ with $\tb(L)< -3$. In the proof of Proposition~\ref{p3} we computed the $d_3$-invariants for these surgeries, though we note that since we are not assuming we have a smooth cosmetic surgery, we do not have the fact that $\tau(L)=0$. Thus they only difference for our current situation is that the rotation number is not bounded as before, that is, the $i$ in the formulas can take more values (prescribed by the Bennequin bound and not the $\tau$-bound). So we must check that distinct contact surgeries on $L$ give distinct contact structures on $\partial X_{-2}$. (Please see the proof of Proposition~\ref{p3} for the notation.) The different contact surgeries on a fixed Legendrian knot correspond to the $\pm$ in the formula above for the $d_3$-invariant in the proof of Proposition~\ref{p3}. It is clear that, as long as $k\not=3$ or $\rot(L)=0$, the $d_3$-invariants are different and thus there are no weakly-cosmetic surgeries. 

We now turn to the case that $\tb(L)=-3$. We notice that smooth surgery with coefficient $-2$ is a contact $(+1)$-surgery and hence there is a unique such surgery, and so there can be no cosmetic contact surgeries with the same smooth surgery coefficient $-2$. 
Now, for a Legendrian knot with $\tb(L)\geq -2$, we see that there are no strongly-cosmetic surgeries with smooth surgery coefficient $-2$ by Item~(1) of this proposition when $\tb(L)>-2$, and when $\tb(L)=-2$, contact surgery is not well-defined. 
\end{proof}

\appendix
\section{Linear algebra computations}\label{lac}

We recall a few facts about matrices. Given an invertible $n\times n$ matrix $A$ with entries $a_{i,j}$, denote its inverse by $B$ with entries $b_{i,j}$. We can compute the entry $b_{i,j}$ as follows
\[
b_{i,j}=(-1)^{i+j}\frac{\det A_{j,i}}{\det A}
\]
where $A_{i,j}$ is the $(i,j)$ minor of $A$, that is the $(n-1)\times (n-1)$ matrix obtained from $A$ by deleting the $i^{th}$ row and $j^{th}$ column, see \cite[Section~0.8.2]{HornJohnson2013}.

The following matrices will frequently appear in our discussion:
\[
I_n=
\begin{bmatrix}
  -2 & -1 & &   \\
  -1 & -2 & \ddots&    \\

    & \ddots & \ddots  &-1\\
   &&-1 & -2   \\
   \end{bmatrix}_{n \times n}
\quad \text{ and } \quad
I'_n=
\begin{bmatrix}
-1&-1&0&\cdots &0\\
 0& -2 & -1 & &   \\
  0&-1 & -2 & \ddots&    \\

  \vdots&  & \ddots & \ddots  &-1\\
  0& &&-1 & -2   \\
   \end{bmatrix}_{n \times n}
\]
One may easily compute that $\det I'_n=-\det I_{n-1}$.
\begin{lemma}
The determinant of $I_n$ is 
\[
\det I_n= (-1)^n(n+1).
\]
\end{lemma}
\begin{proof}
Clearly, $\det I_1=-2$ and $\det I_2=3$. Using the formula
\[
\det I_n=-2\det I_{n-1} -(-1)\det I'_{n-1},
\]
One may easily establish the formula via induction. 
\end{proof}

We will need to use the following lemma.
\begin{lemma}\label{negdefproof}
Given a matrix of the form
\[
M=\begin{bmatrix}
A & B\\
B^T & I_n
   \end{bmatrix}
\]
where $I_n$ is the matrix above
\[
A=\begin{bmatrix}
a&b\\
b&c
   \end{bmatrix}
\quad \text{ and } \quad
B=\begin{bmatrix}
0&0&\ldots &0\\
-1&0&\ldots &0
   \end{bmatrix}
\]
is a $2\times n$ matrix. Then 
\[
\det M = (-1)^n((a(c+1)-b^2)n+(ac-b^2)).
\]
Thus, if $(a(c+1)-b^2)n+ac-b^2$ is positive, then $M$ is negative definite.
\end{lemma}
\begin{proof}
If $D$ is an invertible matrix, then the determinant of a matrix of the form 
$$
\begin{bmatrix}
	A & B  \\
	C & D  \\
\end{bmatrix}
$$
is $\det(D)\det(A-BD^{-1}C),$ where $A$ and $D$ are square matrices of size $i$ and $n-i$ and $B$ is ($i\times n-i$) matrix and $C$ is ($n-i \times i$) matrix, see \cite[Page~114]{AbadirMagnus2005}. From this, we see that 
\[
\det M= \det I_n \det (A-BI^{-1}_nB^T).
\]
One may easily see that $B^TI_nB$ is the matrix $\begin{bmatrix}
	0&0\\ 0& a_{1,1}
\end{bmatrix}$
where $a_{1,1}$ is the upper-left entry in $I_n^{-1}$. From the formula recalled at the beginning of this appendix, we see that $a_{1,1}=\det I_{n-1}/\det I_{n}=-n/(n+1)$. Thus
\[
\det M= (-1)^n (n+1) [(a(c+n/(n+1))-b^2]=(-1)^n((a(c+1)-b^2)n+(ac-b)n)
\]
as claimed.

For the second claim, we recall Sylvester's criterion \cite[Theorem~7.2.5]{HornJohnson2013} says that a matrix $M$ is negative definite if and only if $(-1)^{k}\det M_k>0$ where $M_k$ is the $k^{th}$ principal minor, that is the $k\times k$ submatrix in the upper left corner of $A$. 
\end{proof}

\begin{proof}[Proof of Lemma~\ref{tbleqm3andm2surgery}]
The intersection matrix of $X_{-2}$ is negative definite by Lemma~\ref{negdefproof}.

The rotation vector for $X_{-2}$ is $\mathbf{r}=\begin{bmatrix}i,i\pm 1, 0,\ldots, ,0\end{bmatrix}^T$ where $i= -k+1,-k+3, \ldots, k-1$. Thus, if the intersection matrix is $Q$ and its inverse is $Q'$, then $c_1^2(X_{-n})= i^2q'_{1,1}+ 2i(i\pm 1)q'_{1,2} + (i\pm 1)^2q'_{2,2}$. 
Using the formula at the beginning of the appendix, we see
\begin{align*}
q'_{1,1} &= (-1)^{2}(det Q)^{-1} 
\begin{vmatrix}
 -k-2 & -1 & &   \\
 -1 & -2 & \ddots&    \\
  & \ddots &\ddots & -1  \\
   &&-1 & -2   \\
\end{vmatrix}_{k-3 \times k-3}\\
&=(-1)^k\frac 12  \left( (-k-2) \det I_{n-4}
+
\det I_{k-4}'
\right)
=\frac 12 (-k^2+2k+2)
\end{align*}
One may similarly compute 
\[
q'_{1,2}=\frac 12(k^2-3k) \quad  \text{and } \quad q'_{2,2}=\frac 12 (-k^2+4k-3).
\]
Form this one can easily compute the claimed value for $c_1^2(X_{-2})$. 
\end{proof}	
\begin{proof}[Proof of Lemma~\ref{tbleqm3and2surgery}]
To compute the signature, we show that the intersection matrix $Q$ of $X_{2}$ has only one positive eigenvalue. To achieve this, we recall that the eigenvalues are the zeros of the characteristic polynomial of $Q$. So, we will prove that the characteristic polynomial has only one positive solution. We use Descartes' rule of sign, which says that if a single-variable polynomial with real coefficients has $j$ sign change when ordered by descending variable exponent, then there are $j-2l$ positive roots where $l$ is a non-negative number \cite[pp. 89-93]{Struik1969}. Thus, if there is only one sign change, then there is exactly one positive root. 

We compute the characteristic polynomial using the principle minors. Recall that a principal minor of a $n\times n$ matrix $M$ of size $i$ is the determinant of the $i\times i$ matrix obtained by deleting some number of rows of $Q$ and the corresponding columns. Let $E_i(M)$ denote the sum of principal minors of size $i$. Then, the characteristic polynomial of matrix $M$ is
\[
P_M(\lambda)=\lambda^n-E_1(M) \lambda^{n-1}+\cdots  + (-1)^{n-1} E_{n-1}(M) \lambda + (-1)^n E_n(M),
\] 
see \cite[Section~1.2]{HornJohnson2013}. 

Thus, to show $Q$ has only one positive eigenvalue, we must show that the signs of the $E_i(Q)$ alternate as $i$ increases except at one place. In fact, we will show that the $E_i(Q)$ alternate, except $E_{k+1}(Q)$ and $E_{k+2}(Q)$ will have the same sign.  We prove this inductively. Clearly $E_1(Q)$ is just the trace of $Q$ and so has sign $(-1)^1$. For $E_2(Q)$, we note that the $2\times 2$ principal submatrices are either block matrices (with $1\times 1$ blocks) or $2\times 2$ matrices with all non-zero entries. In the former case, the determinant is the product of two negative numbers and, hence, is positive. In the latter case, the possible matrices are 
 \[
 \begin{bmatrix}
 	-k+1 & -k  \\
 	-k & -k-2 \\
 \end{bmatrix},   
  \begin{bmatrix}
 	-k-2 & -1  \\
 	-1 & -2 \\
 \end{bmatrix}, \text{or }
  \begin{bmatrix}
 	-2 & -1  \\
 	-1 & -2 \\
 \end{bmatrix}
 \] 
 and each of these matrices has a determinant of sign $(-1)^2$.

We now inductively assume that we have shown that all the principle $l\times l$ minors have sign $(-1)^l$ for $l<k$. We now consider the $(l+1)\times (l+1)$ principal minors. Each principal $(l+1)\times (l+1)$ submatrix is either a block matrix with blocks of size $l_1,\ldots, l_m$ such that $l_1+\cdots + l_m=l+1$, or is not. In the former case, the sign of the minor will be the product of the signs of the blocks, that is $(-1)^{l_1}\cdots (-1)^{l_m}=(-1)^{l+1}$ as claimed. In the latter case, we must consider the matrices 
$$
\begin{bmatrix}
	-k+1 & -k &  & &   \\
	-k & -k-2 & -1 & &   \\
	& -1 & -2 & \ddots&    \\
	&  & \ddots &\ddots & -1  \\
	&   &&-1 & -2   \\
\end{bmatrix}_{l+1 \times l+1} or \quad
\begin{bmatrix}
	  -k-2 & -1 & &   \\
	 -1 & -2 & \ddots&    \\
	  & \ddots &\ddots & -1  \\
	   &&-1 & -2   \\
\end{bmatrix}_{l+1 \times l+1}
$$
We may use Lemma~\ref{negdefproof} to show that the first matrix has determinant $(-1)^{l-1}(k^2+(2k-l)-2)$ which clearly has sign $(-1)^{l+1}$ (recall that $l$ is less than $k$). The second matrix is easily seen to be negative definite, so it must have determinant $(-1)^{l+1}$ too by Sylvester's criterion mentioned above. This completes the induction for $l$ up to $k$. When $l$ is $k$, then the inductive step is the same except for the two non-block matrices of size $(k+1)\times (k+1)$. These matrices are as in the previous equations except $l+1=k+1$ in this case. One may easily compute that the determinant of the first matrix is $(-1)^k$, which has the wrong sign, but the determinant of the second matrix is $(-1)^{k+1}(k^2+2k+2)$. So their sum has sign $(-1)^{k+1}$ as desired. Finally, one may compute that $E_{k+2}(Q)=\det Q=(-1)^{k+1}$.

The computation of $c_1^2(X_2)$ is identical to the proof of Lemma~\ref{tbleqm3andm2surgery} except for the $q'_{i,j}$ we have
\[
q'_{1,1}=\frac 12(k^2+2k+2), \quad q'_{1,2}=-\frac 12(k^2+k),  \text{and } \quad q'_{2,2}=\frac 12 (k^2-1).
\]
\end{proof}

\bibliography{references}
\bibliographystyle{plain}
\end{document}